\documentclass[11pt,a4paper]{amsart}
\usepackage[utf8]{inputenc}
\usepackage[english]{babel}
\usepackage[T1]{fontenc}
\usepackage{url}
\usepackage{mathrsfs}
\usepackage{gincltex}
\usepackage{ifthen}

\usepackage{tikz}
\usetikzlibrary{knots}

\newcounter{for_counter} 
\newcommand{\forplus}[4][1]{%
	\setcounter{for_counter}{#2}\ifthenelse{\value{for_counter} < #3}{#4%
		\addtocounter{for_counter}{#1}
		\forplus[#1]{\value{for_counter}}{#3}{#4}}{}}
\newcommand{\cros}[2]{
\draw [line width = 1] (#1, #2) to  [out = 0, in = 180, looseness = 0.7] (1 + #1, 1 +  #2);
\draw [white, double=black, double distance = 1, ultra thick] (#1, 1 + #2) to [out = 0, in = 180, looseness = 0.7] (1 + #1, #2);
}
\newcommand{\crosx}[2]{
\draw [line width = 1] (#1, 1 + #2) to [out = 0, in = 180, looseness = 0.7] (1 + #1, #2);
\draw [white, double=black, double distance = 1, ultra thick] (#1, #2) to  [out = 0, in = 180, looseness = 0.7] (1 + #1, 1 +  #2);
}
\newcommand{\block}[3]{
\forplus{0}{#1}{\cros{#2 + \arabic{for_counter}}{#3}}
}
\newcommand{\blockx}[3]{
\forplus{0}{#1}{\crosx{#2 + \arabic{for_counter}}{#3}}
}
\newcommand{\strip}[1]{
	\strand[rotate around={#1:(0,0)}, line width = 1] (-2.82843, 1) 
		to [out =  0, in = 180, looseness = 1] ( 2.72843, 1.00)
		to [out =  0, in =   0, looseness = 1] ( 2.72843, 0.75)
		to [out =180, in =   0, looseness = 1] (-2.79473, 0.75)
		to [out =180, in = 180, looseness = 1] (-2.79473, 0.50)
		to [out =  0, in = 180, looseness = 1] ( 2.95804, 0.50);
}

\title[Meander diagrams of knots and spatial graphs]{Meander diagrams of knots and spatial graphs: proofs of generalized Jablan--Radovi\'{c} conjectures}
\author{Yury Belousov and Andrei Malyutin}
\thanks{The reported study was funded by RFBR according to the research project n.~17-01-00128~A}
\address{%
Yury Belousov\\
Faculty of Mathematics, National Research University HSE
}
\email{bus99@yandex.ru}
\address{%
Andrei Malyutin\\
St.\,Petersburg Department of 
Steklov Institute of Mathematics\\
St.~Petersburg State University
}
\email{malyutin@pdmi.ras.ru}

\newcommand \R       {\mathbb R}
\newcommand \dd      {\partial}

\newcommand \be     {\begin{equation}}
\newcommand \ee     {\end{equation}}

\newtheorem{thm}{Theorem}
\newtheorem{prop}{Proposition}
\newtheorem*{prop*}{Proposition}
\newtheorem{lem}{Lemma}
\newtheorem*{lem*}{Lemma}

\newtheorem{cor}{Corollary}
\newtheorem*{cor*}{Corollary}

\theoremstyle{definition}
\newtheorem*{defn}{Definition}

\theoremstyle{remark}

\begin{document}

\begin{abstract}
We study decomposition into simple arcs (i.\,e., arcs without self-intersections) for diagrams of knots and spatial graphs. 
In this paper, it is proved in particular that if no edge of a finite spatial graph~$G$ is a knotted loop, 
then there exists a plane diagram~$D$ of~$G$ such that
(i)~each edge of~$G$ is represented by 
a simple arc of~$D$ and 
(ii)~each vertex of~$G$ is represented by a point on the boundary of the convex hull of~$D$. 
This generalizes the conjecture of S.~Jablan and~L.~Radovi\'{c} stating that each knot has a meander diagram, i.\,e., a diagram composed of two simple arcs whose common endpoints 
lie on the boundary of the convex hull of the diagram.
Also, we prove another conjecture of Jablan and Radovi\'{c} stating that each 2-bridge knot has a semi-meander minimal diagram, i.\,e., a minimal diagram composed of two simple arcs.
%
\end{abstract} 
\maketitle
\section{Introduction}

The present paper refers to the classical theory of knots, links, tangles, and spatial graphs in Euclidean 3-space~$\R^3$ and addresses questions related to decomposition of planar diagrams into simple arcs (i.\,e., arcs without self-intersections).
Questions of this kind have been studied, in particular, in~\cite{Hot60, Oza07, RJ15, BM17, Owa18a, Owa18b, E-ZHLN18}. See also~\cite{Hot08, AST11, Gam13a, Gam13b, Mat15}.
In this paper, we consider a triple of such questions that appeared as a triple of conjectures in~\cite{RJ15}.
These conjectures can be formulated as follows.
\begin{itemize}
\item[(C1)] Each knot has a \emph{semi-meander diagram}, i.\,e., a diagram composed of two simple arcs.
\item[(C2)] Each knot has a \emph{meander diagram}, i.\,e., a diagram composed of two simple arcs whose common endpoints lie on the boundary of the convex hull of the diagram.
\item[(C3)] Each 2-bridge knot has a minimal diagram that is semi-meander.
\end{itemize}

The existence of semi-meander and meander diagrams for each knot 
(Conjectures~(C1) and~(C2)) 
has been discovered and independently rediscovered several times by distinct methods and in different terms.
Apparently, the first published proof 
that each knot has a diagram composed of two simple arcs (Conjecture~(C1)) is due to G.~Hotz~\cite{Hot60}, 1960,
who considered the so-called \emph{Arkaden-Faden-Lagen representations} (arcade-string-configurations) for knots.
A~knot in its Arkaden-Faden-Lagen representation has a semi-meander projection.
In~\cite{Hot08}, Hotz mentioned that 
these representations had been suggested to him by Kurt Reidemeister and ``are based on a remark of K.\,F.~Gauss~\cite{Gauss},
that each knot has projections on the plane which can be decomposed in two simple strings, this means strings without doppel points''.
In 1989, G.\,S.~Makanin~\cite{Mak89} obtained a series of results in braid theory implying, in particular, that each knot has a diagram in the form of a closed braid composed of two simple arcs. 
This result of Makanin was rediscovered by J.\,A.~Kneissler~\cite{Kne99} in 1999. 
In 2007, an elegant idea of a much shorter proof for the statement of Conjecture~(C1) has appeared in the preprint~\cite{Oza07} of M.~Ozawa.
A~variation of Ozawa's method is described in detail in~\cite{AST11}.
In addition to the above, the existence of a semi-meander diagram for each knot becomes an obvious corollary of a much more general result about spatial graphs due to R.~Shinjo~\cite{Shi05}, 2005, if we treat a knot as a spatial graph with two vertices and two edges (below, this is explained in more detail).
The existence of meander diagrams for each knot 
(Conjecture~(C2)) has been proved by C.~Adams, R.~Shinjo, and K.~Tanaka (see Sec.~2 and Fig.~2 in~\cite{AST11}).
The paper~\cite{Owa18a} contains another proof (which admits a simplification).
Also, Conjecture~(C2) follows from results of S.~Kinoshita~\cite{Kin87} about $\theta_n$-curves combined with methods of~\cite{LJ01} directly extended to the case of tangles.
In addition to the above, much stronger results are obtained in~\cite{E-ZHLN18}, where it is proved that each knot has so-called \emph{potholder diagrams}, which are a very special case of meander diagrams.

In summary, Conjecture~(C1) is known to be true and has a natural generalization to the case of spatial graphs (this generalization is a corollary of a result due to R.~Shinjo; see~\cite{Shi05} and Theorem~\ref{th:Shinjo} below); Conjecture~(C2) is known to be true and has an impressive strengthening in terms of potholder diagrams (see \cite{E-ZHLN18}); as far as we know, Conjecture~(C3) remained open until recently.

Below, we prove Conjecture~(C3), generalize Conjecture~(C2) to the case of spatial graphs (Theorem~\ref{th:meander-SpGr-4}), and discuss further strengthening of this generalization in terms of potholder-like diagrams. 
In order to formulate Theorem~\ref{th:meander-SpGr-4}, we give definitions related to the concept of spatial graphs.

\subsection*{Spatial graphs}
By a \emph{graph} we mean a 1-dimensional CW~complex:
$0$-cells are its \emph{vertices} and $1$-cells are \emph{edges}.
%
A~\emph{spatial graph} is a subset of~$\R^3$ that is (i)~ambient isotopic to the union of a finite number of straight line segments and (ii)~endowed with the structure of a finite graph, i.\,e., of a finite 1-dimensional CW~complex.
(Note that the same subset of~$\R^3$ can bear distinct graph structures, due to the vertices of valence two.)
A~\emph{loop} in a (spatial) graph is an edge whose closure is a circle. 
A~\emph{knotted loop} in a spatial graph is an edge whose closure is a non-trivial knot.
Two spatial graphs are said to be \emph{equivalent} if they are related by an ambient isotopy preserving the graph structure.

A projection~$p$ of a spatial graph~$G$ onto a plane~$E$ in~$\R^3$ is said to be \emph{regular} if:
(i)~for each point~$x$ in~$E$, at most two points of~$G$ project to~$x$;
(ii)~the number of points $x\in E$ such that 
two points of~$G$ project to~$x$ is at most finite;
such points~$x$ are called \emph{double points} or \emph{crossings};
(iii)~no vertex of~$G$ projects to a crossing; 
(iv)~each crossing~$x\in E$ has a neighborhood $U_x\subset E$ such that the set $p^{-1}(U_x)\cap G$ consists of two straight line intervals. 
A~\emph{diagram} of a spatial graph~$G$ is the plane image of a regular projection of a spatial graph~$G'$ equivalent to~$G$ with additional information of under- and over-crossings in all double points
and with a set of marked points that is the image of the set of vertices of~$G'$.
The marked points in a spatial graph diagram are called the \emph{vertices} of the diagram.
(This marking is not redundant because of vertices of valence two.)
The images of the graph edges will be called the \emph{principal arcs} (or \emph{principal curves}) of the diagram.

We say that a diagram~$D$ of a spatial graph~$G$ is \emph{semi-meander} if all of the principle arcs of~$D$ are simple except those representing edges that are knotted loops, and each of these \emph{exceptional} principal arcs is composed of two simple subarcs. (See examples in Fig.~\ref{fig:sp_graph}(b)--(d).)
We say that a semi-meander diagram~$D$ is \emph{meander} if 
(i)~all of the vertices of~$D$ lie on the boundary of the convex hull of~$D$ and (ii)~each exceptional principal arc of~$D$ is cut into two simple subarcs by a point lying on the boundary of the convex hull of~$D$. (Fig.~\ref{fig:sp_graph}(d) provides an example.) 

In the given terminology, we have the following generalization of Conjecture~(C2).

\begin{thm}
\label{th:meander-SpGr-4}
Each spatial graph has a meander diagram.
\end{thm}

Conjecture~(C2) readily follows from Theorem~\ref{th:meander-SpGr-4} when we convert a knot into a spatial graph with two vertices and two edges.

The paper is organized as follows. 
In Section~\ref{sec:proof2}, we prove Theorem~\ref{th:meander-SpGr-4}.
Section~\ref{sec:remarks} contains several remarks and comments: about algorithmic construction of meander and semi-meander diagrams; about an alternative way of proving Theorem~\ref{th:meander-SpGr-4}; about composing of knots and spatial graphs from plane arcs.
In Section~\ref{sec:potholder}, we discuss generalized potholder diagrams for spatial graphs and a corresponding strengthening of Theorem~\ref{th:meander-SpGr-4}.
In Section~\ref{sec:proof3}, we prove Conjecture~(C3).

\section{Proof of Theorem~\ref{th:meander-SpGr-4}}
\label{sec:proof2}

In~\cite{Shi05}, R.~Shinjo proved the following theorem. 

\begin{thm}[Shinjo~\cite{Shi05}]
\label{th:Shinjo}
Let $G$ be a spatial graph and let $H_1,\dots, H_n$ be spatial subgraphs of~$G$.
Suppose that the intersection $H_i\cap H_j$ is contained in the set of vertices of~$G$ whenever $i\neq j$. Suppose that $D_1,\dots, D_n$ are diagrams of $H_1,\dots, H_n$. 
Then there exists a diagram~$D$ of~$G$ whose restrictions to $H_1,\dots, H_n$ are isotopic to $D_1,\dots, D_n$, respectively.
\end{thm}

Obviously, Theorem~\ref{th:Shinjo} implies that each spatial graph without knotted loops 
has a semi-meander diagram.
Then it follows that each spatial graph has a semi-meander diagram (in order to see this, we create a new vertex of valence two on each edge that is a knotted loop).
Moreover, we observe that such ``adding of a vertex'' to each edge that is a knotted loop reduces the problem of finding a meander diagram to the case of spatial graphs without knotted loops. 
Therefore, in order to prove Theorem~\ref{th:meander-SpGr-4} it is enough to invent a universal procedure that transforms a semi-meander but not meander diagram~$D'$ of a spatial graph without knotted loops 
to a semi-meander diagram~$D''$ representing the same spatial graph as~$D'$
and such that the boundary of the convex hull of~$D''$ contains more diagram's vertices as compared to~$D'$.
The theorem will then follow by induction.

We now pass to a description of such a procedure. 
Let $G$ be a spatial graph without knotted loops and let $D'$ be a semi-meander but not meander diagram of~$G$.
We denote by~$V$ the set of all vertices of~$D'$ and let~$V_{ext}$ be
the subset in~$V$ formed by all those vertices that lie on the boundary of the convex hull of~$D'$.
Since by assumption $D'$ is not meander, it follows that
$V$ has a vertex~$v_0$ that is not in~$V_{ext}$.
Obviously, without loss of generality we can assume that $D'$ is contained in a Euclidean disk~$\mathbb{B}^2$ such that $\dd \mathbb{B}^2\cap D'=V_{ext}$.
Take a point $v\in \dd \mathbb{B}^2\setminus D'$.
Our plan is to transform~$D'$ by relocating~$v_0$ to the position of~$v$.
(The transitions (b)$\to$(c) and (c)$\to$(d) in Fig.~\ref{fig:sp_graph} are examples of such transformations.)


\begin{figure}
\scalebox{0.4}{
\begin{tikzpicture}
\node (v1) at (1,  0){};
\node (v2) at (9,  0){};

\node (p1) at (3.5,  3){};
\node (p2) at (6.5,  1){};
\node (p3) at (5, -1){};
\node (p4) at (3.5,  1){};
\node (p5) at (6.5,  3){};

\node (q1) at (3.5, -3){};
\node (q2) at (6.5, -1){};
\node (q3) at (5,  1){};
\node (q4) at (3.5, -1){};
\node (q5) at (6.5, -3){};

\node (c)  at (5,  0){};
\node (name) at (5, -5) {{\huge(a)}};

\begin{knot}[consider self intersections, end tolerance=1pt, clip width = 5pt]
\strand[ultra thick, red] (v1.center) 
	to [out =  45, in = 180, looseness = 1] (p1.center)
	to [out =   0, in =  90, looseness = 1] (p2.center)
	to [out = -90, in =   0, looseness = 1] (p3.center)
	to [out = 180, in = -90, looseness = 1] (p4.center)
	to [out =  90, in = 180, looseness = 1] (p5.center)
	to [out =   0, in = 135, looseness = 1] (v2.center);
	
\strand[ultra thick, blue] (v1.center) 
	to [out = -45, in = 180, looseness = 1] (q1.center)
	to [out =   0, in = -90, looseness = 1] (q2.center)
	to [out =  90, in =   0, looseness = 1] (q3.center)
	to [out = 180, in =  90, looseness = 1] (q4.center)
	to [out = -90, in = 180, looseness = 1] (q5.center)
	to [out =   0, in =-135, looseness = 1] (v2.center);
	
\strand[ultra thick] (v1.center) 
	to [out =   0, in = -100, looseness = 1.5] (c.center)
	to [out =  75, in =  180, looseness = 1.5] (v2.center);

\flipcrossings{3, 5}
\end{knot}

\draw[fill] (v1) circle (0.2);
\draw[fill] (v2) circle (0.2);

\draw[fill, white] (11.5, 0) circle (0.2);
\draw[fill, white] (-1.5, 0) circle (0.2);

\draw[->, line width = 3pt, white] (6.5, -5) to (6.5, -7);

\end{tikzpicture}
\raisebox{7.4cm}{
\begin{tikzpicture}
\draw[->, line width = 3pt] (0, 0) to (2, 0);
\end{tikzpicture}}
\begin{tikzpicture}
\node (v1) at (1,  0){};
\node (v2) at (9,  0){};

\node (l) at (-0.5,  -0.5){};
\node (r) at ( 10.5,  0.5){};

\node (p1) at (3.5,  3){};
\node (p2) at (6.5,  1){};
\node (p3) at (5,   -1){};
\node (p4) at (3.5,  1){};
\node (p5) at (6.5,  3){};

\node (q1) at (3.5, -3){};
\node (q2) at (6.5, -1){};
\node (q3) at (5,    1){};
\node (q4) at (3.5, -1){};
\node (q5) at (6.5, -3){};

\node (c)  at (5,  0){};
\node (name) at (5, -5) {{\huge(b)}};

\begin{knot}[consider self intersections, end tolerance=1pt, clip width = 5pt]
\strand[ultra thick, red] (v1.center) 
	to [out =  45, in = 180, looseness = 1] (p1.center)
	to [out =   0, in =  90, looseness = 1] (p2.center)
	to [out = -90, in =   0, looseness = 1] (p3.center)
	to [out = 180, in = -90, looseness = 1] (p4.center)
	to [out =  90, in =    0, looseness = 1] (3.5,  2)
	to [out =  180, in = 45, looseness = 1] (1, -1)
	to [out =  225, in = 225, looseness = 1] (0, 0)
	to [out =  45, in =  180, looseness = 1] (3.5,  4)
	to [out =   0, in =  135, looseness = 1] (v2.center);
	
\strand[ultra thick, blue] (v1.center) 
	to [out =  -45, in =  180, looseness = 1] (6.5,  -4)
	to [out =    0, in = -135, looseness = 1] (10, 0)
	to [out =   45, in =   45, looseness = 1] (9, 1)
	to [out = -135, in =    0, looseness = 1] (6.5,  -2)
	to [out =  180, in =  -90, looseness = 1] (q2.center)
	to [out =   90, in =    0, looseness = 1] (q3.center)
	to [out =  180, in =   90, looseness = 1] (q4.center)
	to [out =  -90, in =  180, looseness = 1] (q5.center)
	to [out =    0, in = -135, looseness = 1] (v2.center);
	
\strand[ultra thick] (v1.center) 
	to [out =   0, in =-100, looseness = 1.5] (c.center)
	to [out =  75, in =  180, looseness = 1.5] (v2.center);
	
\flipcrossings{2, 4, 6, 5, 7 }
\end{knot}

\draw[fill] (v1) circle (0.2);
\draw[fill] (v2) circle (0.2);

\draw[fill, white] (11.5, 0) circle (0.2);
\draw[fill, white] (-1.5, 0) circle (0.2);

\draw[->, line width = 3pt] (6.5, -5) to (6.5, -7);
\end{tikzpicture}}

\scalebox{0.4}{

\begin{tikzpicture}
\node (v1) at (-1,  0){};
\node (v2) at (11,  0){};

\node (l) at (-0.5,  -0.5){};
\node (r) at ( 10.5,  0.5){};

\node (p1) at (3.5,  3){};
\node (p2) at (6.5,  1){};
\node (p3) at (5, -1){};
\node (p4) at (3.5,  1){};
\node (p5) at (6.5,  3){};

\node (q1) at (3.5, -3){};
\node (q2) at (6.5, -1){};
\node (q3) at (5,  1){};
\node (q4) at (3.5, -1){};
\node (q5) at (6.5, -3){};

\node (c)  at (5,  0){};
\node (name) at (5, -5) {{\huge(d)}};

\begin{knot}[consider self intersections, end tolerance=0.5pt, clip width = 5pt]
\strand[ultra thick, red] (v1.center)
 	to [out =- 90, in = -90, looseness = 1] (7.5, 0) 
 	to [out =  90, in =   0, looseness = 1] (3.5,  3.3)
 	to [out = 180, in =-135, looseness = 1] (1.3, 0.3) 
	to [out =  45, in = 180, looseness = 1] (p1.center)
	to [out =   0, in =  90, looseness = 1] (p2.center)
	to [out = -90, in =   0, looseness = 1] (p3.center)
	to [out = 180, in = -90, looseness = 1] (p4.center)
	to [out =  90, in =   0, looseness = 1] (3.5,  2)
	to [out = 180, in =  45, looseness = 1] (1, -1)
	to [out = 225, in = 225, looseness = 1] (0, 0)
	to [out =  45, in = 180, looseness = 1] (3.5,  4)
	to [out =   0, in = 135, looseness = 1] (9.2, 0.5)
	to [out = -45, in = 180, looseness = 0.2] (v2.center);
	
\strand[ultra thick, blue] (v1.center) 
	to [out =   0, in = 135, looseness = 0.2] (0.8, -0.5) 
	to [out =  -45, in =  180, looseness = 1] (6.5,  -4)
	to [out =    0, in = -135, looseness = 1] (10, 0)
	to [out =   45, in =   45, looseness = 1] (9, 1)
	to [out = -135, in =    0, looseness = 1] (6.5,  -2)
	to [out =  180, in =  -90, looseness = 1] (q2.center)
	to [out =   90, in =    0, looseness = 1] (q3.center)
	to [out =  180, in =   90, looseness = 1] (q4.center)
	to [out =  -90, in =  180, looseness = 1] (q5.center)
	to [out =    0, in = -135, looseness = 1] (8.7, -0.3) 
	to [out =   45, in =    0, looseness = 1] (6.5,  -3.3)
	to [out =  180, in =  -90, looseness = 1] (2.5, 0) 
	to [out =   90, in =   90, looseness = 1] (v2.center);
	
\strand[ultra thick] (v1.center) 
	to [out =   45, in =  180, looseness = 1] (1.3, -0.2) 
	to [out =    0, in = -100, looseness = 1.5] (c.center)
	to [out =   75, in =  180, looseness = 1.5] (8.7, 0.2) 
	to [out =   0, in = -135, looseness = 1] (v2.center);

\flipcrossings{ 4, 6, 7, 9,11,12, 13,14,15,16,18, 20, 21}
\end{knot}

\draw[fill] (v1) circle (0.2);
\draw[fill] (v2) circle (0.2);
\draw[fill, white] (11.5, 0) circle (0.2);
\draw[fill, white] (-1.5, 0) circle (0.2);
\end{tikzpicture}
\raisebox{5.4cm}{
\begin{tikzpicture}
\draw[->, line width = 3pt] (2, 0) to (0, 0);
\end{tikzpicture}}
\begin{tikzpicture}
\node (v1) at (-1,  0){};
\node (v2) at (9,  0){};

\node (l) at (-0.5,  -0.5){};
\node (r) at ( 10.5,  0.5){};

\node (p1) at (3.5,  3){};
\node (p2) at (6.5,  1){};
\node (p3) at (5, -1){};
\node (p4) at (3.5,  1){};
\node (p5) at (6.5,  3){};

\node (q1) at (3.5, -3){};
\node (q2) at (6.5, -1){};
\node (q3) at (5,  1){};
\node (q4) at (3.5, -1){};
\node (q5) at (6.5, -3){};

\node (c)  at (5,  0){};
\node (name) at (5, -5) {{\huge(c)}};

\begin{knot}[consider self intersections, end tolerance=0.5pt, clip width = 5pt]
\strand[ultra thick, red] (v1.center)
 	to [out =-90, in = -90, looseness = 1] (7.5, 0) 
 	to [out = 90, in =   0, looseness = 1] (3.5,  3.3)
 	to [out = 180, in =-135, looseness = 1] (1.3, 0.3) 
	to [out =  45, in = 180, looseness = 1] (p1.center)
	to [out =   0, in =  90, looseness = 1] (p2.center)
	to [out = -90, in =   0, looseness = 1] (p3.center)
	to [out = 180, in = -90, looseness = 1] (p4.center)
	to [out =  90, in =    0, looseness = 1] (3.5,  2)
	to [out =  180, in = 45, looseness = 1] (1, -1)
	to [out =  225, in = 225, looseness = 1] (0, 0)
	to [out =  45, in =  180, looseness = 1] (3.5,  4)
	to [out =   0, in =  135, looseness = 1] (v2.center);
	
\strand[ultra thick, blue] (v1.center) 
	to [out =   0, in = 135, looseness = 0.2] (0.8, -0.5) 
	to [out =  -45, in =  180, looseness = 1] (6.5,  -4)
	to [out =    0, in = -135, looseness = 1] (10, 0)
	to [out =   45, in =   45, looseness = 1] (9, 1)
	to [out = -135, in =    0, looseness = 1] (6.5,  -2)
	to [out =  180, in =  -90, looseness = 1] (q2.center)
	to [out =   90, in =    0, looseness = 1] (q3.center)
	to [out =  180, in =   90, looseness = 1] (q4.center)
	to [out =  -90, in =  180, looseness = 1] (q5.center)
	to [out =    0, in = -135, looseness = 1] (v2.center);
	
\strand[ultra thick] (v1.center) 
	to [out =   45, in = 180, looseness = 1] (1.3, -0.2) 
	to [out =   0, in = -100, looseness = 1.5] (c.center)
	to [out =  75, in =  180, looseness = 1.5] (v2.center);

\flipcrossings{6,8, 9, 10, 11, 12, 13}
\end{knot}

\draw[fill] (v1) circle (0.2);
\draw[fill] (v2) circle (0.2);
\draw[fill, white] (11.5, 0) circle (0.2);
\draw[fill, white] (-1.5, 0) circle (0.2);
\end{tikzpicture}}
\caption{Spatial graph diagram transformation}
\label{fig:sp_graph}
\end{figure}
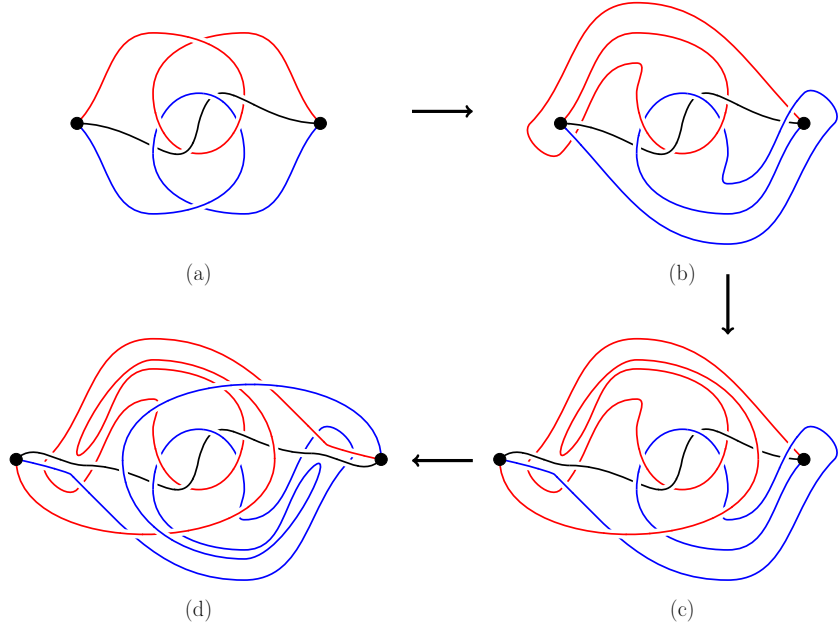

An appropriate transformation can be described as follows: take a small topological closed $2$-disk~$B$ in~$\mathbb{B}^2$ such that 
(i)~$v_0$~is the only vertex of~$D'$ contained in~$B$; 
(ii)~the triple $(B,B\cap D',v_0)$ is homeomorphic to the triple $(B^2,Star_m,0)$, where $B^2$ is a Euclidean 2-disk, $0$ is the center of~$B^2$, $m$ is the valence of~$v_0$, and $Star_m$ is the union of~$m$ radii of~$B^2$. 
The $m$-star\footnote{By an \emph{$m$-star} we mean a topological space homeomorphic to~$Star_m$; the subsets corresponding to the radii of $Star_m$ are called the \emph{legs} of the $m$-star.} $B\cap D'$ is the part of $D'$ that we are going to relocate. 
We introduce additional notation.
Let $p_1, \dots, p_m$ be the points of $\partial B\cap D'$ and let $\alpha_1, \dots, \alpha_m$ denote the legs  of~$B\cap D'$, where
$\alpha_i$ stands for the leg with endpoints at~$p_i$ and~$v_0$. 
Let $A_1, \dots, A_m$ be the principal arcs of~$D'$ containing $p_1, \dots, p_m$, respectively. (We have $A_i=A_j$ whenever for some $i\neq j$ the points $p_i$ and $p_j$ lie on the same looped arc of~$D'$.) 
We observe that since each of $A_1,\ldots,A_m$ is a simple arc having at most one common point with~$\partial\mathbb{B}^2$, it follows that there exists a simple arc~$\beta_i$ with endpoints at~$p_i$ and~$v$ such that (i)~$\beta_i$ does not intersect the arc~$A_i\setminus\alpha_i$; (ii)~$\beta_i$ is contained in~$\mathbb{B}^2$ and $\beta_i\cap\partial\mathbb{B}^2=\{v\}$. 
Also, we observe that if two points $p_i$ and $p_j$ with $i\neq j$ lie on the same principal arc (i.\,e., if $A_i=A_j$) then $\beta_i$ and $\beta_j$ can be chosen such that $\beta_i\cap\beta_j=\{v\}$ in addition to the above conditions (i) and (ii). 
We choose a collection of~$\beta_1$, $\dots$, $\beta_m$ that satisfy all of the above requirements and are in general position both by themselves and with respect to~$D'$. As usual, general position means only double point singularities, at most finite number of double points, in each double point the arcs intersect transversely, no one of $\beta$'s meets a vertex or a crossing of~$D'$, etc. 
Now, let $D''$ be the diagram obtained from~$D'$ by replacing $v_0$ with~$v$ and $\alpha_i$'s with $\beta_i$'s and such that the newly emerged crossings have the following crossing information: (i)~for all crossings between $\beta$'s and arcs of~$D'$, we set that $\beta$'s are over~$D'$; (ii) for all crossings between $\beta_i$ and $\beta_j$ with $i<j$, we let $\beta_j$ go over~$\beta_i$.
We see that~$D''$ is semi-meander, by construction, and 
the set $V_{ext}\cup\{v\}$ is contained in the boundary of the convex hull of~$D''$.
It remains to show that $D''$ represents the same spatial graph as $D'$ does.
We will derive this fact from a mini-theory of \emph{tangle-graphs} developed in the remainder of the present proof.

\subsection*{Tangle-graphs}
Let $\R^3_+$ denote the half-space $\{(x,y,z)\in\R^3\colon z\ge0\}$ in~$\R^3$.
We say that a spatial graph~$T$ in $\R^3$ is a \emph{tangle-graph} if 
(i)~$T$~is contained in~$\R^3_+$ and
(ii)~the intersection $T\cap\partial\R^3_+$ of~$T$ with the plane~$\partial\R^3_+$ is a subset of the set of valence~one vertices in~$T$.
Two tangle-graphs are said to be (\emph{strongly}) \emph{TG-equivalent} if they are related by an isotopy of~$\R^3_+$ 
that fixes~$\partial\R^3_+$ pointwise.
A~\emph{TG-diagram} of a tangle-graph~$T$ is a diagram of~$T$ in the sense of spatial graphs (see definitions above) with additional requirements that it lies in~$\partial\R^3_+$ and is obtained from a tangle-graph that is TG-equivalent to~$T$.

By an \emph{$m$-star-tangle-graph}, where $m$ is a positive integer, we mean a connected tangle-graph of $m+1$ vertices with one vertex of valence~$m$ lying in~$\R^3_+ \setminus \partial\R^3_+$ and $m$ valence~one vertices lying in~$\partial\R^3_+$.
An $m$-star-tangle-graph is \emph{unknotted} if it is TG-equivalent to an $m$-star-tangle-graph all of whose edges are straight line intervals. 
Obviously, two unknotted $m$-star-tangle-graphs with the same set of valence~one vertices are strongly TG-equivalent.
We say that an $m$-star-tangle-graph~$T$ is \emph{monotone} if the normal projection of each edge of~$T$ on the $z$-axis is injective.  

\begin{lem}
\label{lem:monotone}
Each monotone $m$-star-tangle-graph~$T$ is unknotted.
\end{lem}

\begin{proof}
This can be proved by the ``Denne--Sullivan trick'' described in~\cite{DS08}. This trick is an analogue of the ''Alexander trick'' and gives an ambient isotopy that gradually straightens the edges of~$T$ and transforms~$T$ into an $m$-star-tangle-graph all of whose edges are straight line intervals. Denne and Sullivan use this technique when proving the following lemma, which obviously implies Lemma~\ref{lem:monotone}. 
\begin{lem*}[{\cite[Lemma 4.1]{DS08}}]
Let $\mathbb{B}$ be a Euclidean ball centered at $p$, and suppose spatial graphs $\Gamma$ and $\Gamma'$ each consist of $n$ edges starting at $p$ and proceeding out to $\partial \mathbb{B}$ transverse to the concentric nested spheres around $p$. Then $\Gamma$ and $\Gamma'$ are ambient isotopic by an isotopy of $\mathbb{B}$. \qedhere
\end{lem*}
\end{proof}

Now, we apply Lemma~\ref{lem:monotone} to our transformation $D'\to D''$ described above. 
We say that a TG-diagram~$D$ of an $m$-star-tangle-graph~$T$ is \emph{monotone} if (i)~each edge of~$T$ is presented by a simple arc of~$D$ and (ii)~there is an enumeration of edges of~$T$ such that at each crossing of~$D$ the arc representing the edge with the larger number passes over the arc representing the edge with the smaller one.
Clearly, in the diagram~$D''$ above, the subdiagram formed by $\beta$'s can be interpreted as a monotone TG-diagram of an $m$-star-tangle-graph.
The same is trivially true for the subdiagram formed by $\alpha$'s in~$D'$.
It can be easily seen that each $m$-star-tangle-graph with a monotone TG-diagram is strongly TG-equivalent to a monotone $m$-star-tangle-graph.
Consequently, any two $m$-star-tangle-graphs with the same set of valence~one vertices and having monotone TG-diagrams are unknotted (see Lemma~\ref{lem:monotone}) and hence strongly TG-equivalent.
In an obvious way, this implies that $D''$ represents the same spatial graph as $D'$ does, 
which completes the proof of Theorem~\ref{th:meander-SpGr-4}.

\section{Remarks and comments}
\label{sec:remarks}

\subsection{Algorithmic construction of meander and semi-meander diagrams}\label{subsec:algorithmic}
We observe that the above proof of Theorem~\ref{th:meander-SpGr-4} provides an algorithm that transforms a given semi-meander diagram into a meander one. 
However, Theorem~\ref{th:Shinjo} does not give a direct algorithm to produce a semi-meander diagram.
In this connection it should be noted that there exists a simple way to construct semi-meander diagrams for spatial graphs without looped edges. 
Indeed, let $D$ be a diagram of a spatial graph~$G$ without looped edges; if a principal arc~$C$ of~$D$ has self-crossings, we can eliminate a self-crossing of~$C$ that is nearest to an endpoint of~$C$ by a move as in Fig.~\ref{fig:for_rem}; each of newly emerged crossings involves arcs representing distinct edges, so that a finite number of such moves yields a semi-meander diagram.
For the case of knots, the idea of obtaining semi-meander diagrams in this way appears in~\cite{Oza07} (see also~\cite{AST11}). 

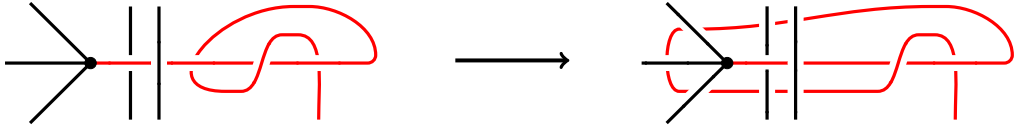
\begin{figure}[h]
\scalebox{0.75}{
\begin{tikzpicture}
\node (v) at (0,  0){};

\node (e1) at (-0.707106781*1.5,  0.707106781*1.5){};
\node (e2) at (-1*1.5,  0){};
\node (e3) at (-0.707106781*1.5,  -0.707106781*1.5){};

\node (p1) at (2.5*2,  0){};
\node (p2) at (1.85*2,  1){};
\node (p3) at (1.25*2, -0.25*2){};
\node (p4) at (1.75*2,  0.5){};
\node (p5) at (2.0*2, -0.5*2){};

\node (a1) at (0.7,  1){};
\node (a2) at (0.7, -1){};

\node (b1) at (1.2,  1){};
\node (b2) at (1.2, -1){};

\begin{knot}[consider self intersections, end tolerance=0.5pt, clip width = 5pt]
\strand[ultra thick] (e1.center)
 	to [out =-45, in = 135, looseness = 0] (v.center);
 	 
\strand[ultra thick] (e2.center)
 	to [out =  0, in = 180, looseness = 0] (v.center);

\strand[ultra thick] (e3.center)
 	to [out = 45, in =-135, looseness = 0] (v.center);

\strand[ultra thick] (a1.center)
 	to [out = -90, in =90, looseness = 0] (a2.center);
 	
\strand[ultra thick] (b1.center)
 	to [out = -90, in =90, looseness = 0] (b2.center);
	
\strand[ultra thick, red, rounded corners] (v.center)  
	to (p1.center)
	to [out =  90, in =    0, looseness = 1] (p2.center)
	to [out = 180, in =  180, looseness = 2] (p3.center)
	to [out =   0, in =  180, looseness = 1] (p4.center)
	to [out =   0, in =   90, looseness = 1] (p5.center);

\flipcrossings{1, 4}
\end{knot}

\draw[fill] (v) circle (0.1);
\draw[fill, white] (6,0) circle (0.1);

\end{tikzpicture}
\raisebox{1.2cm}{
\begin{tikzpicture}
\draw[->, line width = 2pt] (0, 0) to (2, 0);
\end{tikzpicture}}
\begin{tikzpicture}
\node (v) at (0,  0){};

\node (e1) at (-0.707106781*1.5,  0.707106781*1.5){};
\node (e2) at (-1*1.5,  0){};
\node (e3) at (-0.707106781*1.5,  -0.707106781*1.5){};

\node (p1) at (2.5*2,  0){};
\node (p2) at (1.85*2,  1){};
\node (p_p1) at (-0.7,  0.3*2){};
\node (p_p2) at (-0.7, -0.25*2){};
\node (p3) at (1.25*2, -0.25*2){};
\node (p4) at (1.75*2,  0.5){};
\node (p5) at (2.0*2, -0.5*2){};

\node (a1) at (0.7,  1){};
\node (a2) at (0.7, -1){};

\node (b1) at (1.2,  1){};
\node (b2) at (1.2, -1){};

\begin{knot}[consider self intersections, end tolerance=0.5pt, clip width = 5pt]
\strand[ultra thick] (e1.center)
 	to [out =-45, in = 135, looseness = 0] (v.center);
 	 
\strand[ultra thick] (e2.center)
 	to [out =  0, in = 180, looseness = 0] (v.center);

\strand[ultra thick] (e3.center)
 	to [out = 45, in =-135, looseness = 0] (v.center);

\strand[ultra thick] (a1.center)
 	to [out = -90, in =90, looseness = 0] (a2.center);
 	
\strand[ultra thick] (b1.center)
 	to [out = -90, in =90, looseness = 0] (b2.center);
	
\strand[ultra thick, red, rounded corners] (v.center)  
	to (p1.center)
	to [out =  90, in =    0, looseness = 1] (p2.center)
	to [out = 180, in =    0, looseness = 1] (p_p1.center)
	to [out = 180, in =  180, looseness = 1] (p_p2.center)
	to [out =   0, in =  180, looseness = 1] (p3.center)
	to [out =   0, in =  180, looseness = 1] (p4.center)
	to [out =   0, in =   90, looseness = 1] (p5.center);

\flipcrossings{10,  4}
\end{knot}

\draw[fill] (v) circle (0.1);
\draw[fill, white] (-2.5,0) circle (0.1);
\end{tikzpicture}}
\caption{Move eliminating a self-crossing}
\label{fig:for_rem}
\end{figure}

\subsection{An alternative way of proving Theorem~\ref{th:meander-SpGr-4}}
We briefly discuss another method to show that the existence of a semi-meander diagram for a spatial graph implies the existence of a meander one. This alternative way of proof is not constructive (cf.~Section~\ref{subsec:algorithmic}) and based on the phenomenon when an object has a representation with prescribed subrepresentations for parts of the object. 
J.\,H.~Lee and G.\,T.~Jin showed in~\cite{LJ01} that this principle works for the case of links and their diagrams. In~\cite{Shi05}, R.~Shinjo extends their results to the case of spatial graphs (this gives precisely Theorem~\ref{th:Shinjo} above).
The arguments of~\cite{LJ01} deal with the Reidemeister moves for subdiagrams and apply verbatim for tangle diagrams.
Here, by a \emph{tangle} we mean a pair $(B^3,T)$, where $B^3$ is the $3$-ball and $T$ is a compact regular proper $1$-submanifold
in~$B^3$.
A~component~$I$ of a tangle~$(B^3,T)$ is said to be \emph{unknotted} if $I$ lies on a $2$-disk $B^2$ properly embedded in~$B^3$. (For more details on tangles, see~\cite{Kaw96}.)
Applying arguments of~\cite{LJ01} to tangle diagrams immediately yields the following proposition. 

\begin{prop}
\label{pr:meander-tangles}
If $D$ is a diagram of a tangle~$T$, then there exists a diagram~$D'$ of~$T$ such that $D$ and $D'$ are related by a sequence of Reidemeister moves and each unknotted component of~$T$ is represented by a simple arc in~$D'$.
\end{prop}

Proposition~\ref{pr:meander-tangles} implies Theorem~\ref{th:meander-SpGr-4} as follows. When given a semi-meander diagram~$D'$, 
we place~$D'$ inside a Euclidean $2$-disk~$\mathbb{B}^2$ such that~$D'$ does not meet $\dd \mathbb{B}^2$,
then transform~$D'$ into a diagram~$D''$ by moving successively all of the vertices of~$D'$ to~$\dd \mathbb{B}^2$ as in the above prove of Theorem~\ref{th:meander-SpGr-4}, but this time let all of the $\beta$'s for each vertex be straight line segments. 
Now, let $R$ be the radius of $\mathbb{B}^2$ and let $\mathbb{B}^2_{R-\varepsilon}$ denote the Euclidean $2$-disk of radius~$R-\varepsilon$ concentric with~$\mathbb{B}^2$.
If~we choose $\varepsilon>0$ small enough, then $\mathbb{B}^2_{R-\varepsilon}$ contains all of the crossings of~$D''$ and $D''$ imprints a tangle diagram on~$\mathbb{B}^2_{R-\varepsilon}$. We denote this tangle diagram by~$D''_\varepsilon$.
It~can be easily seen that no component of the tangle represented by~$D''_\varepsilon$ is knotted because all of the $\beta$'s are straight line segments. Then Proposition~\ref{pr:meander-tangles} shows that a sequence of Reidemeister moves supported in~$\mathbb{B}^2_{R-\varepsilon}$ turns $D''$ into a meander diagram.

\subsection{Composing knots and spatial graphs of plane arcs}
The fact that each knot has a semi-meander diagram yields the following corollary.
\begin{cor}
Each knot in~$\R^3$ is ambient isotopic to a knot composed of two plane arcs.
\end{cor}
Furthermore, the fact that each knot has a meander diagram yields the following corollary.
\begin{cor}
Each knot in~$\R^3$ is ambient isotopic to a knot~$K$ composed of two plane arcs whose common endpoints lie on the boundary of the convex hull of~$K$.
\end{cor}
Proofs of these two corollaries can be obtained by a direct application of arguments from~\cite{Mat15}. 
To be more specific: for a given knot, we take its arbitrary meander diagram~$D$ composed of two simple arcs $A_1, A_2$ and choose a straight line~$L$ on the plane of the diagram. Then, we denote by $X_1$ (resp., $X_2$) the set of crossings of~$D$ where $A_1$ is over $A_2$ (resp., $A_2$ is over $A_1$) and apply an isotopy of~$D$ that moves $X_1$ and $X_2$ to the different sides of~$L$. After this, one can easily see how a knot with the required property can be obtained. 

These arguments have an obvious generalization to the case of spatial graphs.

\begin{cor}
Each spatial graph without knotted loops is ambient isotopic to a spatial graph all of whose edges are plane and all of whose vertices lie on the boundary of its convex hull.
\end{cor}

\section{Generalized potholder diagrams for spatial graphs}
\label{sec:potholder}
As noted in the introduction, Conjecture~(C2) has an impressive strengthening in terms of potholder diagrams. 
The potholder diagrams of knots are a very specific subclass of meander diagrams, 
and it is shown in~\cite{E-ZHLN18} that each knot has a potholder diagram.
It turns out that Theorem~\ref{th:meander-SpGr-4} can be strengthened in a similar way:
focusing on the results obtained in~\cite{E-ZHLN18}, one can 
generalize the concept of potholder diagrams to the case of spatial graphs and show that each spatial graph has a generalized potholder diagram. 
However, we refrain from giving related definitions and statements in general case because the definitions we obtain are rather cumbersome. 
Instead, we discuss a construction of generalized potholder diagrams in a simple but revealing particular case of spatial graphs having~\emph{odd meander diagrams}.

\begin{defn}
We say that a meander diagram~$D$ of a spatial graph is~\emph{odd} if the number of common crossings is odd for each pair of principal arcs of~$D$.
\end{defn}

The following lemma can be easily deduced from Theorem~\ref{th:meander-SpGr-4}.

\begin{lem}
Any loopless spatial graph with at most three vertices has odd meander diagrams.
\end{lem}

For spatial graphs having odd meander diagrams, we define generalized potholder diagrams as odd meander diagrams of a specific type. The idea of a generalized potholder diagram is given in Fig.~\ref{fig:potholder}. The precise definition is as follows. 

\begin{defn}
We say that an odd meander diagram~$D$ of a spatial graph~$G$ is a \emph{generalized potholder diagram} if there exist a closed Euclidean $2$-disk~$B^2$ with the following properties:
\begin{enumerate}
\item $B^2$ contains all of the crossings of~$D$.
\item $B^2$ contains no vertex of~$D$.
\item Each principal arc of~$D$ is the union of two arcs outside of~$B^2$, several subarcs of~$\partial B^2$, and several parallel chords of~$\partial B^2$. 
\item If two chords of~$\partial B^2$ are parts of distinct principal arcs of~$D$, then these chords do intersect. 
\item If $A_1$, $A_2$, and $A_3$ are tree distinct principal arcs of~$D$, then $A_3$ does not intersect the convex hull of the intersection $A_1\cap A_2$.
\end{enumerate}
\end{defn}

\begin{center}
\begin{figure}[htp]
\begin{tikzpicture}
\path[use as bounding box] (-4,4) rectangle (4,-4.2);
\node (v1) at (0, 4) {};
\node (v2) at (3.464102, -2) {};

\draw[help lines] (0, 0) circle (3);

\begin{knot} [consider self intersections, end tolerance=3pt, clip width = 5pt]
	\strand[line width = 1] (v1.center) 
		to [out =178, in = 180, looseness = 1.4] (-2.82843, 1);
		
	\strand[line width = 1] (v1.center) 
		to [out =  0, in = 60, looseness = 1] (2.28024, 1.949492);
		
	\strand[line width = 1] (v1.center) 
		to [out = 182, in = 120, looseness = 1] (-1.912033, 2.311738);
		
	\strand[line width = 1] (v2.center) 
		to [out = 60, in =   0, looseness = 1] ( 2.95804, 0.50);
		
	\strand[line width = 1] (v2.center) 
		to [out =-118, in = -120, looseness = 1.4] (-1.046007, -2.811738);
		
	\strand[line width = 1] (v2.center) 
		to [out = -122, in = -60, looseness = 1] (0.5481896, -2.9494922);
	
	\strip{0};
	\strip{120};
	\strip{240};

\flipcrossings{1, 3, 6, 8, 9, 11, 12, 15, 17, 19, 22, 26};
\end{knot}

\draw [rotate around={0:(0,0)}, gray, fill, opacity = 0.2, rounded corners] (-2.95804, 0.4) rectangle ( 2.95804 ,1.1);
\draw [rotate around={120:(0,0)}, gray, fill, opacity = 0.2, rounded corners] (-2.95804, 0.4) rectangle ( 2.95804 ,1.1);
\draw [rotate around={240:(0,0)}, gray, fill, opacity = 0.2, rounded corners] (-2.95804, 0.4) rectangle ( 2.95804 ,1.1);

\draw[fill] (v1) circle (0.075);
\draw[fill] (v2) circle (0.075);
\end{tikzpicture}
\caption{A generalized potholder diagram of a spatial graph}
\label{fig:potholder}
\end{figure}
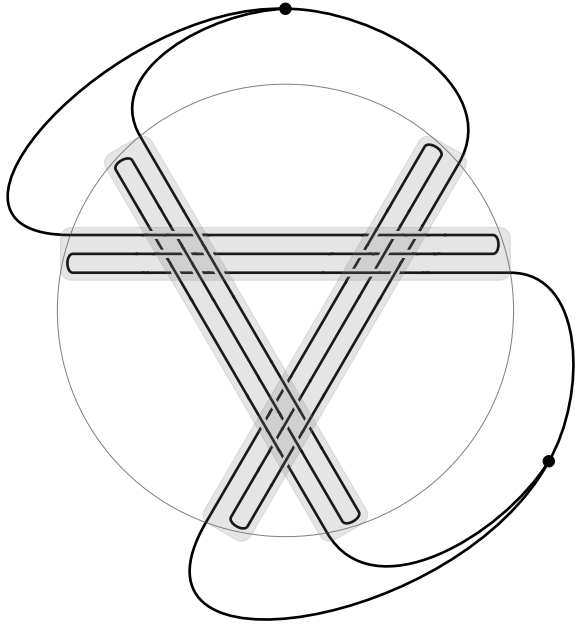
\end{center}

It is convenient to see the structure of a generalized potholder diagram with help of rectangles (``thin'' rectangular strips) as in Fig.~\ref{fig:potholder}.
Indeed, the above definition implies that if~$D$ is a generalized potholder diagram, with the corresponding disc~$B^2$, 
then one can assign a rectangle~$S_A$ to each principal arc~$A$ of~$D$ in such a way that 
(i)~for each principal arc $A$ of~$D$, the intersection $A\cap B^2$ is contained in~$S_A$;
(ii)~no three rectangles corresponding to distinct principal arcs of~$D$ have a common point.

\begin{thm}
\label{thm:potholder}
If a spatial graph has an odd meander diagram, then it has generalized potholder diagrams.
\end{thm}

We describe a procedure of obtaining a generalized potholder diagram
starting with an odd meander diagram.

\subsubsection*{A procedure for obtaining generalized potholder diagrams for spatial graphs.} 
Given an odd meander diagram~$D_0$, we apply a plane isotopy and move $D_0$ to a position where $D_0$ is contained in a Euclidean $2$-disk~$B^2$ such that the intersection $D_0\cap \dd B^2$ is precisely the set of vertices of~$D_0$.
After that, we move the vertices of~$D_0$ outside of~$B^2$ in such a way that
(i)~$B^2$ contains no vertex of~$D_0$;
(ii)~$B^2$ contains all of the crossings of~$D_0$;
(iii)~each principal arc of~$D_0$ intersects~$\dd B^2$ in precisely two points.
Then, for each principal arc~$A$ of~$D_0$, we choose a thin rectangular strip,~$S_A$, such that the interior of~$S_A$ contains the chord~$I_A$ with endpoints at $A\cap\dd B^2$. 
Let~$\mathcal{S}$ be the set of chosen strips for all of the principal arcs of~$D_0$.
Clearly, without loss of generality we can assume that the points of~$D_0\cap\dd B^2$ are in general position in the sense 
that no six points of~$D_0\cap\dd B^2$ are the endpoints of three chords having a common point.
We also assume that the strips in~$\mathcal{S}$ are sufficiently thin such that 
(i)~each strip from~$\mathcal{S}$ intersects $\partial B^2$ in precisely two arcs;
(ii)~no two strips in~$\mathcal{S}$ have a common point outside of the interior of~$B^2$;
(iii)~no three strips in~$\mathcal{S}$ have a common point.
Observe that any two strips in~$\mathcal{S}$ do intersect because~$D_0$ is odd.

We say that a principal arc~$A$ of~$D_0$ is \emph{in a normal position} if the intersection~${A\cap B^2}$ is (i)~contained in the strip~$S_A\in\mathcal{S}$ and (ii)~composed of several subarcs of~$\partial B^2$ and several parallel chords of~$\partial B^2$ (cf.~the requirements specified in the definition of generalized potholder diagrams).
Our further plan is to move the principal arcs of~$D_0$, one by one, to normal positions. 
Let $n$ be the total number of principal arcs of~$D_0$.
We label these arcs arbitrarily with integers $1,\dots,n$ 
and move them into normal positions starting with arcs with smaller labels.
In~what follows, we denote the principal arc with label~$i$ by~$A_i$. 
When moving the next arc~$A_i$ to a normal position, we keep the previous arcs~$A_j$, $j<i$, in normal positions (it may be necessary, however, to increase the number of their ``zigzags'', i.\,e., the number of chords and arcs of~$\dd B^2$ they composed of). 
Also, when moving~$A_i$ to a normal position, it may be necessary to change the arcs $A_k$, $k>i$, that are not in normal position yet. 
While changing these arcs, we can keep them simple with help of the technique described in~\cite{LJ01},
so we pay no attention to them in the subsequent description of an iteration. 
At each iteration, we bring the next principal arc~$A_i$ to a normal position in five steps:
\begin{itemize}
\item[Step 1.] On the $2$-disk~$B^2$, we introduce a \emph{system~$M_i$ of parallels and meridians}, similar to the system of parallels (arcs of latitude) and meridians on a globe projection,
with poles located at the points of~$A_i\cap\dd B^2$. 
The meridians of~$M_i$ are the chord~$I_{A_i}$ (with the endpoints at the poles) and all of the circular arcs that lie in~$B^2$ and have both endpoints at the poles~$A_i\cap\dd B^2$.
The parallels of~$M_i$ are the arcs that lie in~$B^2$, have both endpoints at~$\dd B^2$, and are perpendicular to all meridians.
Subarcs of meridians and parallels will be called \emph{meridional} and \emph{longitudinal} arcs, respectively. By a small isotopy, we transform the arc~$\gamma_i:=A_i\cap B^2$ into a piecewise smooth simple arc composed of a finite number of meridional and longitudinal pieces. 
It is assumed that~$\gamma_i$ will preserve the indicated properties at the following steps of transformation (with a temporary loss of simplicity at Step~4) until the final Step~5.

\item[Step 2.] 
The aim of Step~2 is to make a \emph{polarization} of~$\gamma_i$, i.\,e., to transform $\gamma_i$ to a position where each meridional piece of~$\gamma_i$ is almost as long as the corresponding bearing meridian and all of longitudinal pieces of~$\gamma_i$ are located in small neighborhoods of the poles where no other arc of~$D_0$ enters. 
At~the first stage of the polarization, we ``split'' meridional pieces of~$\gamma_i$ containing more than one crossing of~$D_0$ into several meridional and (tiny) longitudinal pieces such that, finally, no meridional piece of~$\gamma_i$ contains two crossings of~$D_0$. 
In what follows, we say that a subset of~$B^2$ is \emph{elementary} if this subset contains neither two crossings of~$D_0$, nor two meridional pieces of~$\gamma_i$, nor a crossing of~$D_0$ together with a meridional piece of~$\gamma_i$ without crossings.
At the second stage of the polarization, we apply a small isotopy of~$D_0$ in order to shift~$D_0$ in a position where each meridian of~$M_i$ is elementary.
After that, we can find a finite set $\theta$ of meridians of~$M_i$ such that (i)~the meridians in~$\theta$ contain neither crossings of~$D_0$ nor meridional pieces of~$\gamma_i$ and (ii)~each region bounded by two neighboring meridians of~$\theta$ is elementary.
When performing the next stage of the polarization, we use the move that replaces a small subarc of a longitudinal piece of~$\gamma_i$ with a path of the form (meridional piece)-(longitudinal piece)-(meridional piece). We apply these moves simultaneously to families of such small longitudinal subarcs located along the meridians in~$\theta$ and shift all these small longitudinal subarcs towards the same pole of~$M_i$.
In the newly emerged crossings of~$\gamma_i$ with other arcs of~$D_0$, we assume that~$\gamma_i$ goes over these arcs (then the transformed diagram represents the same spatial graph).
Then, we complete the polarization in the regions bounded by neighboring meridians of~$\theta$. 
There are only finite number of region types we have to deal with. The check is elementary. We omit the details.

\item[Step 3.] We push~$\gamma_i$ into the corresponding strip~$S_{A_i}$. 
When performing this transformation, we alternately shift all of the meridional pieces of~$\gamma_i$ into~$S_{A_i}$, one-by-one, starting with those that are closer to~$S_{A_i}$.

Recall that when shifting~$\gamma_i$, we pay no attention to the arcs $A_k$, $k>i$, that wait for their turn.
The arcs waiting for their turn can be adjusted appropriately and kept simple with help of technique described in~\cite{LJ01}.

Observe that no meridional piece of~$\gamma_i$ we are shifting can cling with a principal arc~$A_j$, $j<i$, which is in a normal position, because the chords $I_{A_j}$ and $I_{A_i}$ do intersect whence it follows that each meridian of~$M_i$ intersects each chord contained in~$A_j$ in precisely one point. 
Therefore, the crossings between the arcs in normal positions are the only obstacle that a meridional piece of~$\gamma_i$ can meet when moving towards~$S_{A_i}$.
When the meridional piece of~$\gamma_i$ that we are shifting meets such a crossing, the problem is solved by applying the type III Reidemeister move either directly or after converting the moving meridional piece into a zigzag made of three long meridional pieces and a pair of tiny longitudinal ones.

\item[Step 4.] We bring $\gamma_i$ to a \emph{monotone position} such that no two longitudinal pieces of~$\gamma_i$ intersect the same meridian of~$M_i$ in their internal points.
If we forget for a moment about the other than~$\gamma_i$ arcs of the diagram and about small neighborhoods~$U_1\cup U_2$ of poles, then transforming of $\gamma_i$ to a monotone position is seen as a sequence of transpositions of adjacent meridional pieces. (During this stage of transformation, $\gamma_i$ can have self-intersections in~$U_1\cup U_2$.)
Step~4 does not increase the number of long meridional pieces of~$\gamma_i$. 
On the contrary, each transposition of adjacent meridional pieces of~$\gamma_i$ may require to increase (up to tripling) the number of chords in each of the previously normalized arcs~$\gamma_j$, $j<i$ (these arcs were fixed under Steps 1--3).

\item[Step 5.] We transform $\gamma_i$ from monotone to a normal position in an obvious way. 
\end{itemize}

After bringing all of the arcs of the diagram into normal positions, we can add several ``trivial'' zigzags to each arc, so many as to make their numbers equal for distinct arcs. Besides, the distances between the adjacent chords of each principal arc can be made equal.

We check that after we transform all of the principal arcs $A_1, \dots, A_n$ to normal positions, the resulting diagram satisfies all of the requirements in the definition of generalized potholder diagrams.
Requirements~(1) and~(2) are satisfied because we moved our diagram to the appropriate position at the very beginning of our construction while during the subsequent transformations, all of the changes were made inside~$B^2$, which could not affect the requirements~(1) and~(2).
Requirement~(3) is fulfilled due to the definition of a normal position.
In order to check requirement~(4), we observe that (i)~any two strips in~$\mathcal{S}$ do intersect because~$D_0$ is odd and (ii)~the intersection of any two strips in~$\mathcal{S}$ is a parallelogram contained in the interior of~$B^2$ due to the assumption on the thinness of strips in~$\mathcal{S}$.
Requirement~(5) is fulfilled due to the assumption that no three strips in~$\mathcal{S}$ have a common point.

\section{Proof of Conjecture~(C3)}
\label{sec:proof3}
This section concerns classical knots: no spatial graphs involved.
Hereafter, we~work in the smooth category and use standard terminology of knot theory, which can be found, e.\,g.,~in~\cite{Rol76, Kaw96, BZ03}.
We recall that a diagram~$D$ of a knot~$K$ is said to be \emph{minimal} if no diagram of~$K$ has less number of crossings than that of~$D$.
A simple arc of a knot diagram is called a \emph{bridge} if it has no under-crossings.
A \emph{2-bridge} knot (or \emph{rational} knot; see~\cite{KL04} for details) is a knot that has a diagram with two bridges containing all crossings.

\begin{thm}
\label{th:rational_knot}
Each 2-bridge knot has a minimal diagram that is semi-meander.
\end{thm}
\begin{proof}
We start with auxiliary definitions.
By a \emph{$2m$-plat diagram} (of a knot) we mean a plane knot diagram that is composed of $2m$ \emph{strings}~--- simple arcs each of which is the graph of a smooth function $[0,1]\to[0,1]$. An example of a 4-plat diagram is given in Fig.~\ref{pic:4plat}. It is proved in \cite{KL04} that each 2-bridge knot has a 4-plat diagram that is minimal. Moreover, each 2-bridge knot has a 4-plat minimal diagram with a string free of crossings, but we do not use this fact in our proof to which we now pass.

\begin{figure}[h]
\begin{center}
\scalebox{1.2}{
\begin{tikzpicture}[scale = 0.8]
\draw[line width = 1.2, blue] (2, 0)					 
					to (0, 0);								 
\draw[line width = 1.2, blue] (10, 0) to (4, 0);
\draw[line width = 1] (0, 3) to (4, 3);
\draw[line width = 1] (8, 3) to (10, 3);
\draw[line width = 1] (4, 1) to (5, 1);
\draw[line width = 1] (2, 2) to (4, 2);
\draw[line width = 1] (5, 3) to (6, 3);
\draw[line width = 1] (6, 1) to (8, 1);
\block{2}{0}{1}
\blockx{1}{4}{2}
\block{1}{5}{1}
\blockx{2}{6}{2}
\block{2}{8}{1}

\draw [line width = 1, blue] (-0.5, 0.5) to [out = -90, in = 180,  looseness = 1] (0, 0);
\draw [line width = 1, blue] (10.5, 0.5) to [out = -90, in = 0,  looseness = 1] (10, 0);

\draw [line width = 1] (-0.5, 0.5) to [out = 90, in = 180,  looseness = 1] (0, 1);
\draw [line width = 1] (10.5, 0.5) to [out = 90, in = 0,  looseness = 1] (10, 1);

\draw [line width = 1] (-0.5, 2.5) to [out = -90, in = 180,  looseness = 1] (0, 2);
\draw [line width = 1] (10.5, 2.5) to [out = -90, in = 0,  looseness = 1] (10, 2);
\draw [line width = 1] (-0.5, 2.5) to [out = 90, in = 180,  looseness = 1] (0, 3);
\draw [line width = 1] (10.5, 2.5) to [out = 90, in = 0,  looseness = 1] (10, 3);

\draw [line width = 1] (2, 1) to [out = 0, in = 180, looseness = 0.7] (3, 0);
\draw [white, double=blue, double distance = 1, ultra thick] (2, 0) to  [out = 0, in = 180, looseness = 0.7] (3, 1);
\draw [line width = 1, blue] (3, 1) to [out = 0, in = 180, looseness = 0.7] (4, 0);
\draw [white, double=black, double distance = 1, ultra thick] (3, 0) to  [out = 0, in = 180, looseness = 0.7] (4, 1);

\draw [red, thick, opacity = 0.3, rounded corners] (0,-2.2) rectangle (13.0 ,5.2);
\draw [red, thick, opacity = 0.3, rounded corners] (1,-2.0) rectangle (12.8 ,5.0);
\draw [red, thick, opacity = 0.3, rounded corners] (2,-1.8) rectangle (12.6 ,4.8);
\draw [red, thick, opacity = 0.3, rounded corners] (3,-1.6) rectangle (12.4 ,4.6);
\draw [red, thick, opacity = 0.3, rounded corners] (4,-1.4) rectangle (12.2 ,4.4);
\draw [red, thick, opacity = 0.3, rounded corners] (5,-1.2) rectangle (12.0 ,4.2);
\draw [red, thick, opacity = 0.3, rounded corners] (6,-1.0) rectangle (11.8 ,4.0);
\draw [red, thick, opacity = 0.3, rounded corners] (7,-0.8) rectangle (11.6 ,3.8);
\draw [red, thick, opacity = 0.3, rounded corners] (8,-0.6) rectangle (11.4 ,3.6);
\draw [red, thick, opacity = 0.3, rounded corners] (9,-0.4) rectangle (11.2 ,3.4);
\draw [red, thick, opacity = 0.3, rounded corners] (10,-0.2) rectangle (11.0 ,3.2);

\end{tikzpicture}}
\end{center}
\caption{A 4-plat diagram with a chosen string and a system of special annuli}
\label{pic:4plat}
\end{figure}
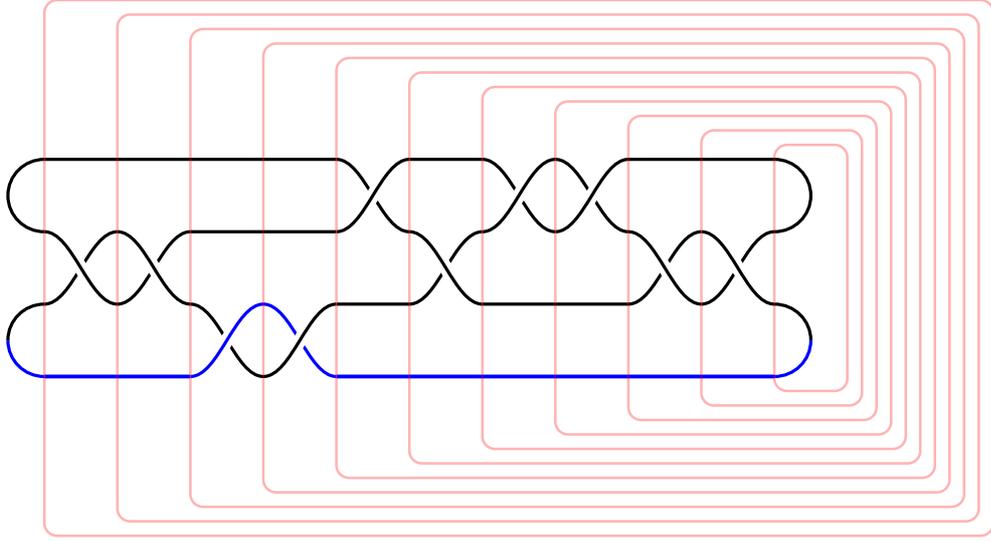
Let $K$ be a 2-bridge knot and let $D\subset\R^2$ be a minimal 4-plat diagram of~$K$.
Without loss of generality we can assume that no two crossings of~$D$ lie on the same vertical line. 
We say that an annulus $A$ in~$\R^2$ is \emph{special} with respect to $D$ if $\partial A$ intersects $D$ transversely, $A$ contains precisely one crossing of~$D$, and for each string~$S$ of~$D$, the intersection of~$A$ with~$S$ is an arc with endpoints on distinct components of~$\partial A$. 
If~$C$ is a simple closed curve in~$\R^2$ intersecting the square  $[0,1]\times[0,1]$ in a vertical segment containing a crossing of~$D$, then a sufficiently small regular neighborhood of~$C$ is obviously a special annulus (see Fig.~\ref{pic:4plat}).
Observe that if $A$ is a special annulus, then the pair $(A,A\cap D)$ looks like those ones shown in Fig.~\ref{pic:flype}.
It is obvious from the picture of Fig.~\ref{pic:4plat} that there exists a system $\mathfrak{A}$ of disjoint special annuli such that each crossing~$x$ of~$D$ is contained in an annulus~$A_x$ of~$\mathfrak{A}$. 
\begin{figure}[h!]
\begin{center}
\begin{tikzpicture}
\draw [line width = 1] (-0.3535, 0.3535) to [out =  150, in = 30, looseness = 1] (-1,-1);
\draw [white, double=black, double distance = 1, ultra thick] (-0.3535,-0.3535) to [out = 210, in = -30, looseness = 1] (-1, 1);

\draw [line width = 1] (0.3535, 0.3535) to [out = 0, in = 180, looseness = 1] (1, 1);
\draw [line width = 1] (0.3535,-0.3535) to [out = 0, in = 180, looseness = 1] (1,-1);

\draw node (R) at (0, 0) {R};
\draw [pink, thick] (0, 0) circle (0.5);
\draw [pink, thick] (0, 0) circle (1.414213562);

\draw (3.5, 0.3) node[scale = 1] {\text{flype}};
\draw[<->, thin] (2.7,0)--(4.3, 0);
\draw [line width = 1] (7+0.3535,-0.3535) to [out = -30, in = 210, looseness = 1] (7+1, 1);
\draw [white, double=black, double distance = 1, ultra thick] (7+0.3535, 0.3535) to [out =  30, in = 150, looseness = 1] (7+1,-1);

\draw [line width = 1] (7-0.3535,0.3535) to [out = 180, in = 0, looseness = 1] (7-1, 1);
\draw [line width = 1] (7-0.3535,-0.3535) to [out = 180, in = 0, looseness = 1] (7-1,-1);

\draw node (R) at (7, 0) [yscale = -1]{R};
\draw [pink, thick] (7, 0) circle (0.5);
\draw [pink, thick] (7, 0) circle (1.414213562);

\end{tikzpicture}
\end{center}
\caption{Flype}
\label{pic:flype}
\end{figure}

Let $S$ be one of the four strings of~$D$. If $S$ contains all crossings of~$D$, then $D$ is semi-meander.
If the set $X$ of crossings in $D\setminus S$ is non-empty, then in order to obtain a minimal semi-meander diagram for~$K$ we will transform~$D$ using Tait's \emph{flypes}.
A~\emph{flype} is a diagram transformation described by the pictures of Fig.~\ref{pic:flype}.
It is well known and can be easily seen that two knot diagrams related by a flype represent the same knot and have the same number of crossings. 
If $x$ is a crossing in $D\setminus S$ and $A$ is a special annulus containing~$x$, 
then performing a flype corresponding to~$A$ eliminates $x$ replacing it with a crossing on~$S$.
We transform $D$ by $|X|$ consecutive flypes performing a flype corresponding to the annulus~$A_x$ in~$\mathfrak{A}$ for each crossing~$x$ from~$X$.
The resulting diagram~$D'$ may no longer be a 4-plat diagram.
However, the four strings of~$D$ by construction yield four simple arcs of~$D'$.
Therefore, $D'$ is a minimal diagram of~$K$, $S$ transforms to a simple arc~$S'$ of~$D'$, and $S'$ contains all crossings of~$D'$.
This means that~$D'$ is a minimal semi-meander diagram of~$K$, as required.
\end{proof}

\subsubsection*{Acknowledgments}
The authors are heartily grateful to M.~Ozawa, R.~Shinjo, G.~Hotz, and A.~Gamkrelidze for valuable information and suggestions.

\end{document}